\documentclass[12pt,reqno]{amsart}

\usepackage{amssymb}

\newtheorem{theorem}{Theorem}[section]
\newtheorem{proposition}[theorem]{Proposition}
\newtheorem{lemma}[theorem]{Lemma}
\newtheorem{corollary}[theorem]{Corollary}

\theoremstyle{definition}
\newtheorem{remark}[theorem]{Remark}

\newcommand{\CO}{\mathcal O}
\newcommand{\PP}{\mathbb P}

\numberwithin{equation}{section}

\begin{document}

\baselineskip=15pt

\title[Principal co-Higgs bundles on $\PP^1$]{Principal co-Higgs bundles on $\PP^1$}

\author[I. Biswas]{Indranil Biswas}

\address{School of Mathematics, Tata Institute of Fundamental
Research, Homi Bhabha Road, Mumbai 400005, India}

\email{indranil@math.tifr.res.in}

\author[O. Garc\'ia-Prada]{Oscar Garc\'ia-Prada}

\address{ICMAT (Instituto de Ciencias Matem\'aticas), C/ Nicol\'as
Cabrera, no. 13--15, Campus Cantoblanco, 28049 Madrid, Spain}

\email{oscar.garcia-prada@icmat.es}

\author[J. Hurtubise]{Jacques Hurtubise}

\address{Department of Mathematics \& Statistics, McGill University, Burnside
Hall, 805 Sherbrooke Street West, Montr\'eal, QC, H3A 0B9, Canada}

\email{jacques.hurtubise@mcgill.ca}

\author[S. Rayan]{Steven Rayan}

\address{Department of Mathematics \& Statistics, University of Saskatchewan, McLean 
Hall, 106 Wiggins Road, Saskatoon, SK, S7N 5E6, Canada}

\email{rayan@math.usask.ca}



\date{}

\begin{abstract}
For complex connected, reductive, affine, algebraic groups $G$, we give a Lie-theoretic characterization of the semistability of 
principal $G$-co-Higgs bundles on the complex projective line $\PP^1$ in terms of the simple roots of a Borel subgroup of $G$. We describe a stratification of the moduli space in terms of the Harder-Narasimhan type of the underlying bundle.
\end{abstract}

\maketitle

\section{Introduction}

Co-Higgs bundles arise principally in the study of generalized holomorphic bundles on generalized complex 
manifolds, in the sense of \cite{H1}, \cite{G1}. On the one hand, a generalized holomorphic bundle on a symplectic manifold 
is a bundle with flat connection. In contrast, a generalized holomorphic bundle on an ordinary complex manifold is 
not simply a holomorphic vector bundle, but rather a holomorphic vector bundle together with Higgs field-like data. 
If $E\,\longrightarrow\, X$ is the vector bundle, then this data is a holomorphic section $\phi$ of
the vector bundle $\mbox{End}(E)\otimes TX$, 
where $TX$ is the holomorphic tangent bundle of $X$, such that $\phi\wedge\phi\,=\,0$, as in \cite{H3,R0}. For 
Higgs bundles in the ordinary sense, $TX$ would be replaced by $T^*X$ --- hence the nomenclature \emph{co-Higgs}.

When $X$ is a complete algebraic curve or equivalently a compact Riemann surface, a co-Higgs
bundle on it is a holomorphic vector bundle 
$E\,\longrightarrow\, X$ together with a section $\phi\,\in\, H^0(X,\,\mbox{End}(E)\otimes K^{-1}_X)$, where 
$K^{-1}_X\,=\, TX$ is the anticanonical line bundle. As such, they provide examples of stable pairs in the sense of \cite{Nit,Mar,Bot}. For curves, the condition $\phi\wedge\phi\,=\,0$ becomes vacuous and slope stability implies that nontrivial examples exist only on the complex projective line $\mathbb P^1$ \cite{R2}. Restricting to $\PP^1$, where $K^{-1}_{\PP^1}\,\cong\, \CO_{\PP^1}(2)$ and holomorphic vector bundles decompose into direct sums of line bundles, one of the main results regarding the existence of semistable co-Higgs bundles is the following, proven by deforming a model co-Higgs field:\\

\begin{theorem}[{\cite[Theorem 6.1]{R2}}]\label{ths}
A holomorphic vector bundle of rank $r$
$$E\,\cong\,\CO_{\PP^1}(m_1)\oplus\CO_{\PP^1}(m_2)\oplus\cdots\oplus\CO_{\PP^1}(m_r)$$ over $\PP^1$, where $m_1\geq m_2\geq\cdots\geq m_r$, admits a 
semistable $\phi\in H^0(\PP^1,\mbox{\emph{End}}(E)\otimes\CO_{\PP^1}(2))$ if and only if $m_i\,\leq\, m_{i+1}+2$ for
all $1\,\leq\, i\,\leq\, r-1$. The 
generic $\phi$ leaves invariant no subbundle of $E$ whatsoever and is hence stable.
\end{theorem}

Just as one generalizes Higgs bundles to principal $G$-Higgs bundles, where $G$ is some reductive affine algebraic group, one can generalize
co-Higgs bundles in the same way. To be precise, we define a \emph{principal $G$-co-Higgs bundle} on $X$ to be a pair
$(E_G,\, \theta)$ in which $E_G\,\longrightarrow\,X$ is a holomorphic principal $G$--bundle and $\theta\,\in\, H^0(X,\,
\mbox{ad}(E_G)\otimes TX)$. The notions of stable and semistable principal bundles extend to
the context of $G$-co-Higgs bundles in a natural way; the definition of a (semi)stable $G$-co-Higgs
bundle is recalled in Section \ref{se2}.

With these notions in place, a natural question is how the above semistability result on $\PP^1$ transforms under the generalization 
to arbitrary reductive $G$. Fix a Borel subgroup $B\, \subset\, G$ and a maximal torus $T\, \subset\, B$,
where $G$ is a connected, reductive, affine algebraic group define over $\mathbb C$. We prove that the that
Harder--Narasimhan reduction of a holomorphic principal $G$--bundle $E_G$ over the complex projective line $\PP^1$ admits a
further holomorphic reduction of structure group to $T$ (Proposition \ref{prop2}).

We prove the following criterion (see Theorem \ref{thm1}):

\begin{theorem}\label{thmi}
Let $X$ stand for the complex projective line
$\PP^1$. Let $E_G$ be a holomorphic principal $G$--bundle $E_G$ on $X$, and let $E_T\, \subset\, E_G$ be
a holomorphic reduction of structure group to $T$ of the Harder--Narasimhan reduction of
$E_G$ (as mentioned above). There is a co-Higgs field
$$
\theta\, \in\, H^0(X,\, {\rm ad}(E_G)\otimes TX)
$$
such that $(E_G,\, \theta)$ is stable if and only if for every simple root
$$\chi\, :\, T\, \longrightarrow\, {\mathbb C}^*$$
with respect to $(B,\, T)$, the inequality
$${\rm degree}(L_\chi)\, \leq\, 2$$ holds, where 
$L_\chi$ is the line bundle on $X$ associated, via $\chi$, to $E_T$. In fact, if
${\rm degree}(L_\chi)\, \geq\, 3$ for some simple root $\chi\, :\, T\, \longrightarrow\, {\mathbb C}^*$
with respect to $B$, then there is no co-Higgs field
$$
\theta\, \in\, H^0(X,\, {\rm ad}(E_G)\otimes TX)
$$
such that $(E_G,\, \theta)$ is semistable.
\end{theorem}

When $G\, =\, \text{GL}(r, {\mathbb C})$, Theorem \ref{thmi} coincides with Theorem \ref{ths}.

\begin{remark}\label{rem1}
From the openness of the stability condition, \cite{Ma}, it follows immediately that if there is a co-Higgs field
$\theta_0$ on $E_G$ such that $(E_G,\, \theta_0)$ is stable, then for the generic co-Higgs field $\theta$ on $E_G$,
the pair $(E_G,\, \theta)$ is stable.
\end{remark}

\begin{remark}
Note that in Theorem \ref{ths}, it is necessary to take the degrees $\{m_j\}_{j=1}^r$ in decreasing order, otherwise the
criterion is not valid. Similarly, the criterion in Theorem \ref{thmi} is not valid if we take $E_T$ to be an arbitrary
holomorphic reduction of structure group of $E_G$ to the maximal torus $T$. The principal
$T$--bundle $E_T$ in Theorem \ref{ths} has to be a
reduction of structure group of the Harder--Narasimhan reduction of $E_G$.
\end{remark}

The proof of Theorem \ref{thmi} uses an auxiliary object --- namely the adjoint co-Higgs bundle as defined below --- together with a study of the 
Harder-Narasimhan filtration under the reduction of structure to a maximal torus.

Finally, for arbitrary $G$ satisfying the above hypotheses, we describe in Section \ref{SectStrata} a stratification of the corresponding co-Higgs moduli space in terms of the Harder-Narasimhan type of the underlying bundle.

\section{Adjoint bundle of a co-Higgs bundle}\label{se2}

Let $G$ be a connected, reductive, affine algebraic group defined over $\mathbb C$. The Lie algebra of $G$ will
be denoted by $\mathfrak g$. Recall that a \textit{parabolic subgroup} of $G$ is a Zariski-closed, connected
subgroup $P\, \subset\, G$ such that $G/P$ is compact. The unipotent radical of a parabolic
subgroup $P\, \subset\, G$ will be denoted by $R_u(P)$. The quotient group $L(P)\,:=\, P/R_u(P)$ is called the
\text{Levi quotient} of $P$. A \textit{Levi subgroup} of $P$ is a Zariski closed subgroup $L_P\, \subset\, P$ such that
the composition
$$
L_P\, \hookrightarrow\, P\, \longrightarrow\, L(P)\,:=\, P/R_u(P)
$$
is an isomorphism \cite[p. 158, \S~11.22]{Bo}. We recall that any parabolic subgroup $P$ admits a 
Levi subgroup, and any two Levi subgroups of $P$ differ by conjugation by an 
element of $R_u(P)$ \cite[p. 158, \S~11.23]{Bo}, \cite[\S~30.2, p. 184]{Hu}.
The center of $G$ will be denoted by $Z_G$. A 
character $\chi$ of $P$ is called \textit{strictly anti-dominant} if
\begin{itemize}
\item $\chi$ is trivial on $Z_G$, and

\item the line bundle on $G/P$, associated to the principal $P$--bundle $G\,
\longrightarrow\, G/P$ for the character $\chi$, is ample. 
\end{itemize}

For economy, we will understand $X$ to refer to the complex projective line 
${\mathbb P}^1$. Accordingly, its holomorphic tangent bundle will be denoted by 
$TX$. A co-Higgs field on a holomorphic vector bundle $V$ over $X$ is a holomorphic 
section of $\text{End}(V)\otimes TX$. Take a holomorphic principal $G$--bundle $E_G$ 
on $X$. Let
$$
\text{ad}(E_G)\,:=\, E_G\times^G {\mathfrak g}\, \longrightarrow\, X
$$
be the adjoint vector bundle associated to $E_G$ for the adjoint action of $G$ on $\mathfrak 
g$. A \textit{co-Higgs field} on $E_G$ is a holomorphic section of $\text{ad}(E_G)\otimes 
TX$. A $G$-co-Higgs bundle is a pair $(E_G,\, \theta)$, where $E_G$ is a holomorphic
principal $G$--bundle and $\theta$ is a co-Higgs field on $E_G$.

Note that a co-Higgs field on a principal $\text{GL}(r, {\mathbb C})$--bundle $E$ on 
$X$ is a co-Higgs field on the vector bundle of rank $r$ associated to $E$ by the standard 
representation of $\text{GL}(r, {\mathbb C})$.

Let $\theta$ be a co-Higgs field on a holomorphic principal $G$--bundle $E_G$ on $X$. Then $\theta$ induces a co-Higgs
field $\text{ad}(\theta)$ on the vector bundle $\text{ad}(E_G)$ which is defined as follows:
$$
\text{ad}(\theta)(w)\,=\, [\theta(x),\, w]\, , \ \ \ \forall \ \ w\, \in\, \text{ad}(E_G)_x\, ,\, \ x\,\in\, X\, .
$$

A $G$-co-Higgs bundle $(E_G,\, \theta)$ on $X$ is called
\textit{stable} (respectively, \textit{semistable})
for every pair $(P,\, E_P)$, where
\begin{itemize}
\item $P\, \subsetneq\, G$ is a parabolic subgroup, and

\item $E_P\, \subset\, E_G$ is a holomorphic reduction of structure group to $P$ such that
$$
\theta\, \in\, H^0(X, \,\text{ad}(E_P)\otimes TX)\, ,
$$
\end{itemize}
the following holds: for every strictly anti-dominant character $\chi$ of $P$, the 
line bundle $E_P\times^{\chi}\mathbb C$ on $X$ associated to $E_P$, for the character $\chi$, is of
positive (respectively, nonnegative) degree. This definition coincides with stability of principal $G$-Hitchin pairs taking values in an arbitrary line bundle, cf. \cite{GGM,BGM,BP,Sch}. It is known that any holomorphic principal $G$--bundle $E_G$ has a Harder--Narasimhan
reduction $E_Q\, \subset\, E_G$, where $Q$ is a parabolic subgroup of $G$; this Harder--Narasimhan is unique in a precise
sense, which is recalled in Section \ref{se3}.

\begin{proposition}\label{prop1}
Let $(E_G,\, \theta)$ be a semistable $G$-co-Higgs bundle on $X$. Then the co-Higgs vector bundle
$({\rm ad}(E_G),\, {\rm ad}(\theta))$ is semistable.
\end{proposition}

\begin{proof}
Assume that the co-Higgs vector bundle $(\text{ad}(E_G),\, \text{ad}(\theta))$ is not semistable. Let
\begin{equation}\label{f1}
0\,=\, V_0\,\subset\, V_1\,\subset\, \cdots \,\subset\, V_{\ell-1}\,\subset\, V_\ell \,=\, \text{ad}(E_G)
\end{equation}
be the Harder--Narasimhan filtration of $(\text{ad}(E_G),\, \text{ad}(\theta))$. We recall that this means that
\begin{itemize}
\item each $V_i$ is a holomorphic subbundle of $\text{ad}(E_G)$;

\item $\text{ad}(\theta)(V_i)\, \subset\, V_i$ for all $i$;

\item the vector bundles $V_i/V_{i-1}$, $1\, \leq\, i\, \leq\, \ell$, equipped with the co-Higgs
field induced by $\text{ad}(\theta)$ is semistable; and

\item the inequalities
$$
\frac{\text{degree}(V_1)}{\text{rank}(V_1)} \, >\, \frac{\text{degree}(V_2/V_1)}{\text{rank}(V_2/V_1)}
\, > \, \cdots \, > \, \frac{\text{degree}(V_{\ell-1}/V_{\ell-2})}{\text{rank}(V_{\ell-1}/V_{\ell-2})}
 \, > \, \frac{\text{degree}(V_{\ell}/V_{\ell-1})}{\text{rank}(V_{\ell}/V_{\ell-1})}
$$
hold.
\end{itemize}
The proof of the existence and uniqueness of the above Harder--Narasimhan filtration is identical to the proof
of the existence and uniqueness of the Harder--Narasimhan filtration of a vector bundle.
Note that $\text{End}(\text{ad}(E_G))\,=\, \text{ad}(E_G)\otimes\text{ad}(E_G)^*\,=\,
\text{End}(\text{ad}(E_G)^*)$. Using this isomorphism, the co-Higgs field
$\text{ad}(\theta)$ on $\text{ad}(E_G)$ induces a co-Higgs field on the dual vector bundle
$\text{ad}(E_G)^*$. Let
\begin{equation}\label{e1}
\text{ad}(\theta)^*\, \in\, H^0(X,\, \text{End}(\text{ad}(E_G)^*)\otimes TX)
\end{equation}
be the co-Higgs field on $\text{ad}(E_G)^*$ induced by $\text{ad}(\theta)$. It may be mentioned that
the dual co-Higgs vector bundle $(\text{ad}(E_G)^*,\, \text{ad}(\theta)^*)$ is not semistable because
$(\text{ad}(E_G),\, \text{ad}(\theta))$ is not semistable. In fact the Harder--Narasimhan
filtration for $(\text{ad}(E_G)^*,\, \text{ad}(\theta)^*)$ is
\begin{equation}\label{f2}
0\,=\, W_0\,\subset\, W_1\,\subset\, \cdots \,\subset\, W_{\ell-1}\,\subset\, W_\ell \,=\, \text{ad}(E_G)^*\, ,
\end{equation}
where $W_i$ is the annihilator $V^\perp_{\ell-i}$ of $V_{\ell-i}$ for every $0\, \leq\, i\, \leq\, \ell$. Indeed,
this follows directly from the defining properties of the Harder--Narasimhan filtration and
the fact that the dual of semistable co-Higgs bundle is also semistable.

Fix a non-degenerate symmetric bilinear form $B_0$ on $\mathfrak g$ such that the adjoint action of $G$ on
$\mathfrak g$ preserves $B_0$. Since $B_0$ is $G$--invariant, it produces a fiber-wise non-degenerate 
symmetric bilinear form
$$
B'\, :\, \text{ad}(E_G)\otimes \text{ad}(E_G)\,\longrightarrow\, {\mathcal O}_X\, .
$$
In particular, $B'$ produces a holomorphic isomorphism of vector bundles
\begin{equation}\label{e4}
{\mathcal B}\, :\, \text{ad}(E_G)\,\longrightarrow\, \text{ad}(E_G)^*\, .
\end{equation}
Consider the composition
$$
\text{ad}(E_G)^*\,\stackrel{{\mathcal B}^{-1}}{\longrightarrow}\, \text{ad}(E_G)
\,\stackrel{\text{ad}(\theta)}{\longrightarrow}\, \text{ad}(E_G)\otimes TX
\,\stackrel{{\mathcal B}\otimes {\rm Id}_{TX}}{\longrightarrow}\, \text{ad}(E_G)^*\otimes TX\, .
$$
Evidently it is a co-Higgs field on the vector bundle $\text{ad}(E_G)^*$. Let
\begin{equation}\label{e2}
\text{ad}(\theta)'\, \in\, H^0(X,\, \text{End}(\text{ad}(E_G)^*)\otimes TX)
\end{equation}
be this co-Higgs field on $\text{ad}(E_G)^*$.

We will investigate how the co-Higgs field $\text{ad}(\theta)^*$ on $\text{ad}(E_G)^*$ constructed
in \eqref{e1} is related to the co-Higgs field $\text{ad}(\theta)'$ constructed in \eqref{e2}.

Let ${\mathfrak z}({\mathfrak g})$ be the center of $\mathfrak g$, so
${\mathfrak z}({\mathfrak g})$ is the Lie algebra of the center $Z_G$ of $G$.
Since $G$ is reductive, the $G$--module $\mathfrak g$ has a natural decomposition
\begin{equation}\label{e3}
{\mathfrak g}\,=\, {\mathfrak z}({\mathfrak g})\oplus \left(\bigoplus_{i=1}^d {\mathfrak h}_i\right)\, ,
\end{equation}
where each ${\mathfrak h}_i$ is an ideal in $\mathfrak g$ with the Lie algebra
${\mathfrak h}_i$ being simple. The decomposition in \eqref{e3} is
orthogonal with respect to the earlier mentioned
symmetric bilinear form $B_0$ on $\mathfrak g$ because $B_0$ is
preserved by the adjoint action of $G$. Since ${\mathfrak h}_i$ is simple, if
$B_i$ and $\widetilde{B}_i$ are two two $G$--invariant nonzero
symmetric bilinear forms on ${\mathfrak h}_i$, then there is a nonzero constant
$\lambda\, \in\, \mathbb C$ such that $\widetilde{B}_i\,=\, \lambda\cdot B_i$. On the other
hand, a simple Lie algebra has the Killing form which is nondegenerate. From these it follows that
$\text{ad}(\theta)'$ in \eqref{e2} does not depend on the choice of $B_0$. More precisely, it
coincides with $\text{ad}(\theta)^*$ in \eqref{e1}.

Since the isomorphism ${\mathcal B}$ in \eqref{e4} takes the co-Higgs bundle $(\text{ad}(E_G),\,\text{ad}(\theta))$
to the co-Higgs bundle $(\text{ad}(E_G)^*,\,\text{ad}(\theta)')$, it takes the Harder--Narasimhan filtration
of $\text{ad}(E_G)$ for $(\text{ad}(E_G),\,\text{ad}(\theta))$ to the Harder--Narasimhan filtration
of $\text{ad}(E_G)^*$ for $(\text{ad}(E_G)^*,\,\text{ad}(\theta)')$. It was observed above that the co-Higgs field
$\text{ad}(\theta)'$ coincides with $\text{ad}(\theta)^*$. Consequently, the isomorphism ${\mathcal B}$ in \eqref{e4} takes
the filtration of $\text{ad}(E_G)$ in \eqref{f1} to the filtration of $\text{ad}(E_G)^*$ in \eqref{f2}.

The tensor product of two semistable co-Higgs vector bundles is again
semistable \cite[p. 2263,
Theorem 1.1]{BP}. Therefore, the proof of Proposition 2.10 in \cite[p. 214]{AB} now goes through.
\end{proof}

\begin{corollary}\label{cor1}
Let $E_G$ be a holomorphic principal $G$--bundle on $X$ that admits a
co-Higgs field such that the resulting $G$-co-Higgs bundle is semistable. Then
the vector bundle ${\rm ad}(E_G)$ admits a
co-Higgs field such that the resulting co-Higgs bundle is semistable.
\end{corollary}

\begin{proof}
If $\theta$ is a co-Higgs field on $E_G$ such that the $G$-co-Higgs bundle $(E_G,\,
\theta)$ is semistable, then Proposition \ref{prop1} says that the co-Higgs bundle
$({\rm ad}(E_G),\, {\rm ad}(\theta))$ is semistable.
\end{proof}

Note that from Theorem \ref{ths} we know when the vector bundle ${\rm ad}(E_G)$ admits a
co-Higgs field such that the resulting co-Higgs bundle is semistable.

\section{Torus reduction of principal $G$--bundles}\label{se3}

Let $E_G$ be a holomorphic principal $G$--bundle on $X$. Let
\begin{equation}\label{e5}
E_P\, \subset\, E_G
\end{equation}
be the Harder-Narasimhan reduction for $E_G$ \cite{AAB}. It should be clarified that we are not
assuming that $E_G$ is not semistable. If $E_G$ is semistable, then
$P\,=\, G$. It should be mentioned that $E_G$ determines the conjugacy class of the
parabolic subgroup $P$ in \eqref{e5}. Indeed, for any $g_0\,\in\, G$, the sub-fiber bundle
$E_Pg_0\, \subset\, E_G$ is a holomorphic reduction of structure group of $E_G$ to
$g^{-}_0Pg_0$. This reduction of structure group $E_Pg_0$ of $E_G$ to $g^{-}_0Pg_0$ is also
Harder-Narasimhan reduction for $E_G$. The Harder-Narasimhan reduction is unique in the
sense that any two Harder-Narasimhan reductions differ, in the above way, by some element of $G$.

Fix a Borel subgroup $B\, \subset\, P$ of $G$, and also fix a maximal
torus $T\, \subset\, B$. Given any parabolic subgroup of $Q\, \subset\ G$, there is an
element $g_0\,\in\, G$ such that $g^{-1}_0Qg_0\, \supset\, B$ \cite[Theorem 21.3]{Hu}. We assume that $P$ in \eqref{e5}
contains $B$.

The following lemma is basically a consequence of the splitting theorem of Birkhoff and Grothendieck \cite{Bi,Gr}.

\begin{lemma}\label{lem1}
If $E'_P$ is a holomorphic principal $P$--bundle $E'_P$ on $X\,=\, {\mathbb P}^1$, then $E'_P$ admits a holomorphic reduction
of structure group to the Borel subgroup $B$. In particular, the
holomorphic principal $P$--bundle $E_P$ in \eqref{e5} admits a holomorphic
reduction of structure group to $B$.
\end{lemma}

\begin{proof}
As before, the unipotent radical of $P$ is denoted by $R_u(P)$; let $L(P)\,:=\, P/R_u(P)$
be the Levi quotient. Let $$q\, :\, P\, \longrightarrow\,
P/R_u(P)\,=\, L(P)$$ be the quotient map. Note that $q(T)$ is a maximal torus of
$L(P)$. Let
$$
E'_{L(P)}\, := \, E'_P/R_u(P)
$$
be the holomorphic principal $L(P)$--bundle on $X\,=\ {\mathbb P}^1$ obtained
by extending the structure group of $E'_P$ using the quotient map $q$. Let
$$
q_1\, :\, E'_P \, \longrightarrow\,E'_P/R_u(P)\,=\, E'_{L(P)}
$$
be the quotient map.

Since $L(P)$
is a connected reductive complex algebraic group, 
the holomorphic principal $L(P)$--bundle $E'_{L(P)}$ admits a holomorphic reduction of 
structure group to the maximal torus $q(T)\, \subset\ L(P)$ \cite[p.~122,
Th\'eor\`eme 1.1]{Gr}. Let
$$
E'_{q(T)}\, \subset\, E'_{L(P)}
$$
be a holomorphic reduction of structure group to $q(T)$. Now we note that the inverse
image
$$
q^{-1}_1 (E'_{q(T)})\, \subset\, E'_P
$$
is a holomorphic reduction of structure group, to the subgroup $q^{-1}(q(T))\, \subset\,
P$, of the holomorphic principal $P$--bundle $E'_P$. On the other hand, we have
$$
q^{-1}(q(T))\, \subset\, B\, .
$$
Let $E'_B$ be the holomorphic principal $B$--bundle on $X$ obtained by extending the
structure group of the above holomorphic principal $q^{-1}(q(T))$--bundle $q^{-1}_1 (E'_{q(T)})$
using the inclusion of $q^{-1}(q(T))$ in $B$. This $E'_B$ is evidently a
holomorphic reduction of structure group of $E'_P$ to the subgroup $B$ of $P$.
\end{proof}

A stronger statement holds for the reduction $E_P$ in \eqref{e5}.

\begin{proposition}\label{prop2}
The holomorphic principal $P$--bundle $E_P$ in \eqref{e5} admits a holomorphic
reduction of structure group to the maximal torus $T\, \subset\, P$.
\end{proposition}

\begin{proof}
Fix a holomorphic reduction of structure group
\begin{equation}\label{br}
E_B\, \subset\, E_P \, \longrightarrow\, X\, ,
\end{equation}
which exists by Lemma \ref{lem1}.

The unipotent radical of $B$ will be denoted by $R_u(B)$.
Consider the short exact sequence of algebraic groups
\begin{equation}\label{she}
e\, \longrightarrow\, R_u(B) \, \longrightarrow\, B \, \stackrel{\phi}{\longrightarrow}\, T_0\, :=\, B/R_u(B)
\, \longrightarrow\, e\, .
\end{equation}
Note that the restriction $\phi\vert_T \, :\, T\, \longrightarrow\, T_0$ is an isomorphism.
We will often identify $T$ with $T_0$ using $\phi\vert_T$. The short exact sequence
in \eqref{she} is right-split \cite[p. 158, \S~11.23]{Bo}, \cite[\S~30.2, p. 184]{Hu}. Fix a homomorphism
\begin{equation}\label{hs}
s\, :\, T_0\,=\, T\, \longrightarrow\, B
\end{equation}
such that $\phi\circ s\,=\, \text{Id}_{T_0}$.

Consider the holomorphic principal $B$--bundle $E_B$ in \eqref{br}. 
Let
$$
E_{T_0}\, :=\, E_B\times^\phi T_0\, \longrightarrow\, X
$$
be the holomorphic principal $T_0$--bundle obtained by extending the structure group of $E_B$
using the projection $\phi$ in \eqref{she}. Let $E_T\, \longrightarrow\, X$ be the
holomorphic principal $T$--bundle given by $E_{T_0}$ using the isomorphism $\phi\vert_T$. Let
\begin{equation}\label{hsl}
E'_B\, :=\, E_{T_0}\times^s B \, \longrightarrow\, X
\end{equation}
be the holomorphic principal $B$--bundle obtained by extending the structure group of $E_{T_0}$
using the homomorphism $s$ in \eqref{hs}. Note that $E_{T_0}$ is a holomorphic reduction of
structure group of the principal $B$--bundle $E'_B$ to $T_0$, because $E'_B$ is the extension of
structure group of $E_{T_0}$.

Therefore, to prove the proposition it suffices to show that the principal $B$--bundle $E'_B$
constructed in \eqref{hsl} is holomorphically isomorphic to the principal $B$--bundle $E_B$.

Note for for both $E_B$ and $E'_B$, the holomorphic principal $T_0$--bundle obtained by extending 
the structure group using the projection $\phi$ in \eqref{she} is $E_{T_0}\,=\, E_T$. We will 
use this fact in proving that $E_B$ and $E'_B$ are holomorphically isomorphic. For that we 
need to understand the class of principal $B$--bundles that give the same principal 
$T$--bundle $E_{T}$ by extension of structure group using $\phi$.

The subgroup $T\, \subset\, B$ acts on $B$ via inner automorphisms. This action of $T$ on $B$ evidently
preserves the unipotent radical $R_u(B)$ in \eqref{she}; in fact, any automorphism of
$B$ preserves $R_u(B)$. Any $t\, \in\, T$ acts on $R_u(B)$ as
$u\, \longmapsto\, tut^{-1}$, $u\, \in\, R_u(B)$. There is a filtration of algebraic subgroups
\begin{equation}\label{fl}
R_u(B)\,=\, U_0\, \supsetneq\, U_1\, \supsetneq\, U_2\, \supsetneq\, \cdots \, \supsetneq\, U_n \, \supsetneq\, U_{n+1}\,=\, e
\end{equation}
such that
\begin{itemize}
\item $U_i$ is a normal subgroup of $R_u(B)$ for all $0\, \leq\, i\, \leq\, n+1$, 

\item the quotient $U_i/U_{i+1}$ is isomorphic to the additive group
${\mathbb C}$ for all $0\, \leq\, i\, \leq\, n$, and

\item the adjoint action of $T$ on $B$ preserves the filtration in \eqref{fl}.
\end{itemize}
To construct the filtration in \eqref{fl} consider the central series for $R_u(B)$. Each successive quotient of this
central series is a direct sum of copies of the additive group ${\mathbb C}$. The adjoint action of $T$ on $R_u(B)$
clearly preserves the central series for $R_u(B)$. So we may decompose each successive quotient of the
central series as a direct sum of one-dimensional $T$--modules. This way we get a finer filtration of $R_u(B)$ that
satisfies all the conditions for \eqref{fl}.

Let $$E_T(U_i)\,=\, E_T\times^T U_i\, \longrightarrow\, X$$ be the fiber bundle
associated to the principal $T$--bundle $E_T$ for the
adjoint action of $T$ on $U_i$. So each fiber of $E_T(U_i)$ is a group algebraically isomorphic
to $U_i$. Next we define the twisted cohomology set $H^1(X,\, E_T(U_i))$;
they are explicitly described below. The reason for looking into this twisted cohomology is that
the set $H^1(X,\, E_T(U_0))$ parametrizes all the isomorphism classes of holomorphic principal
$B$--bundles on $X$ that produce the same holomorphic principal $T$--bundle $E_T$ after extension
of structure group using $\phi$ in \eqref{she}.

Cover $X$ using two affine open subsets which we denote by $X_1$ and $X_2$; they can
be, for example, the complements of $0$ and $\infty$ respectively. Trivialize
$E_T$ over $X_1$ and $X_2$. Let $\tau_{12}$ be the corresponding $1$--cocycle, which is a map from
$X_{12}\,:=\,X_1\bigcap X_2$ to $T$.

Take any holomorphic principal $B$--bundle
$$E^0_B\, \longrightarrow\, X$$
such that the principal $T$--bundle $E^0_B\times^\phi T$ obtained by extending the
structure group of $E^0_B$ using $\phi$ is $E_T$. Then a cocycle giving $E^0_B$ is of the form
$$
u_{12} \tau_{12}\, :\, X_{12}\,:=\, X_1\cap X_2 \, \longrightarrow\, B\, ,
$$
where $u_{12}\, :\, X_{12} \, \longrightarrow\, U_0\,=\, R_u(B)$. This $u_{12}$ is called
a twisted cocycle. Any two twisted cocycles $u^1_{12}$ and $u^2_{12}$ are called {\it equivalent}
if there are holomorphic maps
$z_i\, :\, X_i\, \longrightarrow\, R_u(B)$, $i\,=\,1,\, 2$, such that
$$
z_1 u^1_{12}\tau_{12}(z^{-1}_2)\,:=\, z_1 u^1_{12}\tau_{12}z^{-1}_2\tau^{-1}_{12}
\,=\, u^2_{12}\, .
$$
The equivalence classes of twisted cocycles form the pointed set
$H^1(X,\, E_T(U_0))$ mentioned earlier (see \cite[Appendix]{FM}). The base point in
$H^1(X,\, E_T(U_0))$ corresponds to the constant function $u_{12}\,=\, e$ (the identity
element of $G$ is denoted by $e$). As mentioned before,
the elements of $H^1(X,\, E_T(U_0))$ correspond to the isomorphism classes of
holomorphic principal $B$--bundles that give $E_T$ by extension of structure group using
$\phi$ \cite[Appendix]{FM}, \cite{BV}. The base point in $H^1(X,\, E_T(U_0))$ corresponds to $E'_B$
constructed in \eqref{hsl}; it was noted earlier that $E'_B$ produces $E_T$ by extension of
structure group using the homomorphism $\phi$.

To prove that $E_B$ and $E'_B$ are isomorphic, it suffices to prove that
\begin{equation}\label{sp}
H^1(X,\, E_T(U_0))\,=\, \{e\}\, ,
\end{equation}
because $E_B$ gives $E_T$ by extension of structure group using $\phi$.

Recall that the action of $T$ on $R_u(B)$ preserves the filtration in \eqref{fl}.
Therefore, $T$ acts on each successive quotient 
$U_i/U_{i+1}$, $0\, \leq\, i\, \leq\, n$. So we can similarly construct
$H^1(X,\, E_T(U_i))$ and $H^1(X,\, E_T(U_i/U_{i+1}))$, for all $0\, \leq\, i\, \leq\, n$.
The inclusion of $U_{i+1}$ in $U_i$ produces a map of pointed sets
$$
\beta_i\, :\, H^1(X,\, E_T(U_{i+1}))\, \longrightarrow\, H^1(X,\, E_T(U_i))\, ;
$$
similarly, the quotient map
\begin{equation}\label{mq}
q_i\,:\,U_i\, \longrightarrow\, U_i/U_{i+1}
\end{equation}
produces a map of pointed sets
$$
\gamma_i\, :\, H^1(X,\, E_T(U_i))\, \longrightarrow\, H^1(X,\, E_T(U_i/U_{i+1}))
$$
\cite[Appendix]{FM}, \cite{BV}.

We claim that the image of $\beta_i$ coincides with the inverse image, under the map
$\gamma_i$, of the base point of $H^1(X,\, E_T(U_i/U_{i+1}))$.

To prove the claim, let $$e_0\, \in\, H^1(X,\, E_T(U_i/U_{i+1}))$$
be the base point. Since the composition
$$
U_{i+1}\, \hookrightarrow\, U_i \, \stackrel{q_i}{\longrightarrow}\, U_i/U_{i+1}\, ,
$$
where $q_i$ is constructed in \eqref{mq}, is the trivial homomorphism, it follows immediately that
$$
\gamma_i(\beta_i(H^1(X,\, E_T(U_{i+1})))) \,=\, e_0\, .
$$
So to prove the claim it suffices to show that
\begin{equation}\label{cp}
\gamma^{-1}_i(e_0)\, \subset\, \beta_i(H^1(X,\, E_T(U_{i+1})))\, .
\end{equation}

Let $u_{12}$ be a twisted cocycle with values in $U_i$ such that the
element $\widetilde{u}_{12}$ in $H^1(X,\, E_T(U_i))$ defined by $u_{12}$ lies in
$\gamma^{-1}_i(e_0)$. Note that the element $\gamma_i(\widetilde{u}_{12})\, \in\,
H^1(X,\, E_T(U_i/U_{i+1}))$ is given by the cocycle $q_i\circ u_{12}$, where $q_i$
is the projection in \eqref{mq}. Therefore, there are morphisms
$$
z_i\, :\, X_{i}\, \longrightarrow\, U_i/U_{i+1}\, , \ \ i\,=\, 1,\, 2\, ,
$$
such that
$$
z_1(q_i\circ u_{12})\tau_{12}(z^{-1}_2)\,:=\, z_1(q_i\circ u_{12})\tau_{12}z^{-1}_2\tau^{-1}_{12}
\,=\, e_i\, ,
$$
where $e_i$ denotes the identity element of $U_i/U_{i+1}$.

Since the varieties $U_i$ and $U_i/U_{i+1}$ are both affine spaces, and $q_i$ is a projection
map of affine spaces, there are algebraic morphisms
$$
\widetilde{z}_i\, :\, X_{i}\, \longrightarrow\, U_i\, , \ \ i\,=\, 1,\, 2\, ,
$$
such that $z_i\,=\, q_i\circ \widetilde{z}_i$ for $i\,=\, 1,\, 2$. Now define
$$
\widehat{u}_{12}\,:=\,
\widetilde{z}_1(q_i\circ u_{12})\tau_{12}(\widetilde{z}^{-1}_2)\,:=\,
\widetilde{z}_1(q_i\circ u_{12})\tau_{12}\widetilde{z}^{-1}_2\tau^{-1}_{12}\, .
$$
It is straight-forward to check that $\widehat{u}_{12}$ maps $X_{12}\,:=\, X_1\cap X_2$
to $U_{i+1}$. The element
$$
\widehat{u}'\,\in\, H^1(X,\, E_T(U_{i+1}))
$$
defined by $\widehat{u}_{12}$ satisfies the condition that $$\beta_i (\widehat{u}')\,=\,
\widetilde{u}_{12}\, .$$ This proves the inclusion in \eqref{cp}. Hence the claim is also proved.

We will prove \eqref{sp} using induction: It will be shown that 
\begin{equation}\label{sp2}
H^1(X,\, E_T(U_i))\,=\, \{e\}
\end{equation}
for all $0\, \leq\, i\, \leq\, n+1$. From the above claim it follows that \eqref{sp2}
holds if
\begin{equation}\label{sp3}
H^1(X,\, E_T(U_i/U_{i+1}))\,=\, \{e\}
\end{equation}
for all $0\, \leq\, i\, \leq\, n$.

Recall that the quotient $U_i/U_{i+1}$ is isomorphic to the additive group $\mathbb C$. Note that
the action of $T$ on $U_i/U_{i+1}$ is linear. Let
$L_i$ denote the holomorphic line bundle on $X$ associated to the principal $T$--bundle
$E_T$ for the action of $T$ on $U_i/U_{i+1}$.

Now from the properties of the Harder--Narasimhan reduction it follows that
$$
\text{degree}(L_i)\, \geq\, 0
$$
\cite[p.~694, Theorem 1]{AAB}.
This implies that $H^1(X,\, L_i)\,=\, 0$. Hence \eqref{sp3} holds. Therefore, we conclude that \eqref{sp2}, in
particular \eqref{sp}, holds. This completes the proof of the proposition.
\end{proof}

\section{Criterion for stable co-Higgs field}

Let $E_G$ be a holomorphic principal $G$--bundle over $X\,=\, {\mathbb P}^1$.
Let
\begin{equation}\label{re}
E_P\, \subset\, E_G
\end{equation}
be the Harder--Narasimhan reduction; if $E_G$ is semistable,
then $P\, =\, G$. As before, fix $B$ and $T$; assume that $B\, \subset\, P$.

Let
$$
E_T\, \subset\, E_B\, \subset\, E_P
$$
be holomorphic reductions (see Lemma \ref{lem1} and Proposition \ref{prop2}).

Let $\chi\, :\, T\, \longrightarrow\, {\mathbb C}^*$ be a simple root with respect to
$(B,\, T)$. Then the associated line bundle
$$
L_\chi\, :=\, E_T\times^\chi {\mathbb C}\, \longrightarrow\, X
$$
satisfies the following condition:
\begin{equation}\label{at6}
\text{degree}(L_\chi)\, \geq\, 0
\end{equation}
\cite[p.~712, Theorem 6]{AAB}.

\begin{theorem}\label{thm1}
There is a co-Higgs field
$$
\theta\, \in\, H^0(X,\, {\rm ad}(E_G)\otimes TX)
$$
such that $(E_G,\, \theta)$ is stable if and only if for every
simple root
$$\chi\, :\, T\, \longrightarrow\, {\mathbb C}^*$$
with respect to $(B,\, T)$, the inequality
$${\rm degree}(L_\chi)\, \leq\, 2$$ holds. In fact, if
${\rm degree}(L_\chi)\, \geq\, 3$ for some simple root $\chi\, :\, T\, \longrightarrow\, {\mathbb C}^*$
with respect to $B$, then there is no co-Higgs field
$$
\theta\, \in\, H^0(X,\, {\rm ad}(E_G)\otimes TX)
$$
such that $(E_G,\, \theta)$ is semistable.
\end{theorem}

\begin{proof}
First assume that there is a simple root 
$$\chi\, :\, T\, \longrightarrow\, {\mathbb C}^*$$
with respect to $(B,\, T)$ such that
\begin{equation}\label{c1}
{\rm degree}(L_\chi)\, \geq\, 3\, .
\end{equation}
Let $B\, \subset\, P_\chi\, \subset\, G$ be the maximal parabolic subgroup corresponding to $\chi$. Let
$$
E_{P_{\chi}}\, :=\, E_B\times^B P_\chi\, \longrightarrow\, X
$$
be the holomorphic principal $P_\chi$--bundle obtained by extending the structure group
of $E_B$ using the inclusion of $B$ in $P_\chi$. Since $E_B$ is a reduction of structure
group of $E_G$, it follows that there is a
natural inclusion $$E_{P_{\chi}}\,=\, E_B\times^B P_\chi \, \hookrightarrow\,
E_B\times^B G \,=\, E_G\, ,$$ meaning $E_{P_{\chi}}$ is a
holomorphic reduction of structure group of $E_G$ to the subgroup $P_\chi\, \subset\, G$.

The quotient vector bundle $\text{ad}(E_G)/\text{ad}(E_{P_{\chi}})$ splits into a direct sum
of line bundles \cite[p.~122, Th\'eor\`eme 1.1]{Gr}. But we need a representation theoretic
description of such a decomposition. For that consider the $T$--module ${\mathfrak g}/\text{Lie}(P_\chi)$;
note that since $T\, \subset\, P_\chi$, the adjoint action of $T$ on ${\mathfrak g}$ preserves
the Lie algebra $\text{Lie}(P_\chi)$. Express $T$--module ${\mathfrak g}/\text{Lie}(P_\chi)$ as a direct
sum of one-dimensional $T$--modules
\begin{equation}\label{dl}
{\mathfrak g}/\text{Lie}(P_\chi)\,=\,\bigoplus_{\mu\in S} {\mathbb L}_\mu\, ,
\end{equation}
where $S$ is a collection of characters of $S$ (there may be repetitions of some
characters), and ${\mathbb L}_\mu$ is the one-dimensional $T$--module on which $T$ acts via $\mu$.
Let
$$
L^\mu\, :=\, E_T\times^\mu {\mathbb L}_\mu\, \longrightarrow\, X
$$
be the holomorphic line bundle associated to the principal $T$--bundle $E_T$ for the $T$--module ${\mathbb L}_\mu$.
So, from \eqref{dl} we have a decomposition
\begin{equation}\label{dl2}
\text{ad}(E_G)/\text{ad}(E_{P_{\chi}})\,=\, \bigoplus_{\mu\in S}L^\mu\, .
\end{equation}

From \eqref{c1} it follows that the
degree of every direct summand $L^\mu$ in \eqref{dl2} is bounded above
by $-3$. Hence we have
$$
H^0(X,\, L^\mu\otimes TX)\,=\, 0\, .
$$
Consequently, from \eqref{dl2} it follows that
$$
H^0(X,\, (\text{ad}(E_G)/\text{ad}(E_{P_{\chi}}))\otimes TX)\,=\, 0\, .
$$
Therefore, for every co-Higgs field $\theta$ on $E_G$, we have
$$
\theta\, \in\, H^0(X,\, \text{ad}(E_{P_{\chi}})\otimes TX)\, .
$$
This implies that the above reduction
$$
E_{P_{\chi}}\,\subset\, E_G
$$
contradicts the condition required for semistability of $(E_G,\, \theta)$.

We will now prove that if ${\rm degree}(L_\chi)\, \leq\, 2$
for every simple root $\chi$ with respect to $(B,\, T)$, then
there is a co-Higgs field
$\theta\, \in\, H^0(X,\, {\rm ad}(E_G)\otimes TX)$ such that $(E_G,\, \theta)$ is stable. 

So assume that
\begin{equation}\label{c2}
{\rm degree}(L_\chi)\, \leq\, 2
\end{equation}
for every simple root $\chi$ with respect to $(B,\, T)$. The condition in \eqref{c2} implies that
\begin{equation}\label{d2}
H^0(X,\, L^*_\chi\otimes TX) \, \not=\, 0\, .
\end{equation}

Note that on $X\,=\, {\mathbb P}^1$, the Harder--Narasimhan filtration of any
holomorphic vector bundle splits holomorphically. Indeed, any semistable vector bundle
on $X$ of rank $a$ and slope $b$ is isomorphic to ${\mathcal O}_{{\mathbb P}^1}(b)^{\oplus a}$.
If $b'\, \geq\, b$, then
$$
H^1({\mathbb P}^1,\, \text{Hom}({\mathcal O}_{{\mathbb P}^1}(b),\,
{\mathcal O}_{{\mathbb P}^1}(b')))\,=\, H^1({\mathbb P}^1,\, {\mathcal O}_{{\mathbb P}^1}(b'-b))
\,=\, 0\, .
$$
Consequently, any short exact sequence of holomorphic vector bundles of the form
\begin{equation}\label{es}
0\, \longrightarrow\, {\mathcal O}_{{\mathbb P}^1}(b')^{\oplus a'}\, \longrightarrow\, F
\, \longrightarrow\, {\mathcal O}_{{\mathbb P}^1}(b)^{\oplus a} \, \longrightarrow\, 0
\end{equation}
splits holomorphically. From this it follows immediately that the Harder--Narasimhan filtration of any
holomorphic vector bundle on $X$ splits holomorphically. In particular, the
Harder--Narasimhan filtration of $\text{ad}(E_G)$ splits holomorphically. The adjoint vector bundle
$\text{ad}(E_P)$ for $E_P$ (in \eqref{re}) is the part of the Harder--Narasimhan filtration of $\text{ad}(E_G)$
for slope zero \cite[p.~216, Lemma 2.11]{AB}. So a splitting of the
Harder--Narasimhan filtration of $\text{ad}(E_G)$ produces a splitting of $\text{ad}(E_P)$ as a direct sum
of semistable vector bundles (of different slopes).

Let $\mathfrak b$ denote the Lie algebras of $B$. Fix a $T$--submodule ${\mathbb W}\, \subset\, \mathfrak b$ such that
\begin{equation}\label{es3}
{\mathfrak b}\,=\, {\mathbb W}\oplus \left(\bigoplus_{\chi\in S_B} {\mathbb L}_\chi\right)\, ,
\end{equation}
where $S_B$ is the set of simple roots with respect to $(B,\, T)$, and ${\mathbb L}_\chi$, as before, is the
submodule of ${\mathfrak b}$ corresponding to $\chi$. Consider the holomorphic vector
bundles on $X$ associated to the holomorphic principal
$T$--bundle $E_T$ for the $T$--modules in \eqref{es3}. Let $W\,:=\, E_T\times^T {\mathbb W}$ be the holomorphic vector bundle
associated to $E_T$ for the $T$--module $\mathbb W$. As before, the holomorphic line bundle on $X$ associated to $E_T$ for the
$T$--module ${\mathbb L}_\chi$ will be denoted by $L_\chi$. Now from \eqref{es3} we have the following:
$$
\text{ad}(E_B)\,=\, W\oplus \left(\bigoplus_{\chi\in S_B} L_\chi\right)\, .
$$

We have $\text{degree}(L_\chi)\, \geq\, 0$ for every $\chi\,\in\, S_B$ (see \eqref{at6}),
so
$$
\text{degree}(L_\chi\otimes TX)\,\,=\,\text{degree}(L_\chi)+2 \, .
$$

Let $B^1$ be the opposite Borel for $(B,\, T)$; so $B^1$ is the unique Borel subgroup of $G$ such that
$B\bigcap B^1\,=\, T$. Let ${\mathfrak b}^1$ denote the Lie algebra of $B^1$.
A $G$--invariant nondegenerate symmetric bilinear form on $\mathfrak g$ produces an isomorphism of ${\mathfrak b}^1$
with ${\mathfrak b}^*$. Therefore, from \eqref{es3} we have an isomorphism of $T$--modules
\begin{equation}\label{es4}
{\mathfrak b}^1\,=\, {\mathbb W}^*\oplus \left(\bigoplus_{\chi\in S_B} {\mathbb L}^*_\chi\right)\, .
\end{equation}

Let $E_{B^1}\, :=\, E_T\times^T B^1\, \longrightarrow\, X$ be the holomorphic principal $B^1$--bundle obtained by
extending the structure group of the principal $T$--bundle
$E_T$ using the inclusion of $T$ in $B^1$. Since $E_T$ is a reduction of structure group
of $E_G$, it follows immediately that $E_{B^1}$ is a reduction of structure group of $E_G$ to the subgroup $B^1\, \subset \, G$.
In particular, $\text{ad}(E_{B^1})$ is a holomorphic subbundle of $\text{ad}(E_G)$.
Consider the holomorphic vector bundles on $X$ associated to the holomorphic principal
$T$--bundle $E_T$ for the $T$--modules in \eqref{es4}. From \eqref{es4} we have the following:
\begin{equation}\label{es5}
\text{ad}(E_{B^1})\,=\, W^*\oplus \left(\bigoplus_{\chi\in S_B} L^*_\chi\right)\, .
\end{equation}

For every $\chi\, \in\, S_B$, fix a nonzero section
$$
0\, \not=\, s_\chi\, \in\, H^0(X,\, L^*_\chi\otimes TX)\, ,
$$
which exists by \eqref{d2}. Now using \eqref{es5}, we have
$$
\bigoplus_{\chi\in S_B}s_\chi\, \in\, H^0(X,\, \text{ad}(E_{B^1})\otimes TX)
\, \subset\, H^0(X,\, \text{ad}(E_G)\otimes TX)\, .
$$
Consequently,
$$\theta\, :=\, \bigoplus_{\chi\in S_B} s_\chi$$
is a co-Higgs field on $E_G$.

We will prove that the $G$-co-Higgs bundle $(E_G,\, \theta)$ constructed above is stable.
This can be found in \cite{HiG}. More precisely, $(E_G,\, \theta)$ is stable for the same
reason as the $G$-Higgs bundle in \cite[p.~456, (5.2)]{HiG} is stable.

We also give an alternative argument for the stability of $(E_G,\, \theta)$. Let $E_{P'}$ be a reduction 
to another parabolic $P'$, and suppose that $\theta$ is a section of ${\rm 
ad}(E_{P'})\otimes TX$. The bundle of Lie algebras ${\rm ad}(E_{P'})$ is a sub-algebra 
bundle of ${\rm ad}(E_{G})$; it has a reduction ${\rm ad}(E_{B'})$ to a Borel, which contains 
a trivial bundle ${\mathfrak t} (E_{B'})$ corresponding to the torus of $B'$. Now we see how 
this interacts with the $P$-bundle $E_P$ given by the Harder-Narasimhan reduction; under 
projection to the quotient of Lie algebra bundles $E_P(\mathfrak g/\mathfrak p)$ (a sum of 
negative line subbundles) ${\mathfrak t} (E_{B'})$ maps to zero, and so lies in ${\rm ad}(E_{P})$; 
in addition it projects with a trivial kernel to the torus subbundle of ${\rm ad}(E_{B})$, the 
Borel reduction of ${\rm ad}(E_P)$. The bundle ${\mathfrak t} (E_{B'})$ then acts non-trivially on 
the element $\theta$, ensuring that each individual simple root space of $S_B$ lies in 
${\rm ad}(E_{P'})$, mod elements of $ad(E_P)$. Taking commutators, the root spaces of ${\rm ad}(E_{B'})$ 
project surjectively to the negative bundle $E_P(\mathfrak g/\mathfrak p)$, and so the anti 
dominant characters of $P'$ have to yield positive line bundles.
\end{proof}

\subsection{An example}

Let $\Omega_0$ be the standard symplectic form on ${\mathbb C}^{2r}$. Let $G\, =\, \text{Sp}(2r,{\mathbb C})$
denote the group of all linear automorphisms of ${\mathbb C}^{2r}$ preserving $\Omega_0$. Giving a holomorphic
principal $G$--bundle $E_G$ on $X$ is equivalent to giving a pair $(V,\, \varphi)$, where $V$ is a
holomorphic vector bundle on $X$ of rank $2r$, and
$$
\varphi\, :\, V\, \otimes V\, \longrightarrow\, {\mathcal O}_X
$$
is a fiber-wise nondegenerate skew-symmetric bilinear form. The Harder--Narasimhan reduction
$E_P\, \subset\, E_G$ of $E_G$ is simply the Harder-Narasimhan filtration
$$
0\, =\, V_0\, \subset\, V_1\, \subset\, V_2 \, \subset\, \cdots\, \subset\, V_{\ell-1} \, \subset\, V_\ell\,=\, V
$$
of the corresponding vector bundle $V$. Note that the annihilator, for the form $\varphi$, of the subbundle $V_i$
is $V_{\ell-i}$. Indeed, this follows immediately from the defining properties of a Harder-Narasimhan filtration
and the fact that the dual of a semistable vector bundle is semistable.

A holomorphic reduction of structure group $E_G$ to a maximal torus $T\, \subset\, G$ is a holomorphic filtration of 
subbundles
$$
0\, =\, W_0\, \subset\, W_1\, \subset\, W_2 \, \subset\, \cdots\, \subset\, W_{2r-1} \, \subset\, W_{2r}\,=\, V
$$
such that for every $0\, \leq\, i\, \leq\, 2r$, the subbundle $W_{2r-i}$ is the annihilator, for the form $\varphi$, of the
subbundle $W_i$. Note that the rank of $W_i$ is $i$, and $W_r$ is a Lagrangian subbundle of $V$.

Let $E_G$ be a holomorphic principal $G$--bundle on $X\, =\, {\mathbb P}^1$. Let
$(V,\, \varphi)$ be the corresponding vector bundle with symplectic form. Then we have a holomorphic decomposition into a
direct sum of holomorphic line bundles
\begin{equation}\label{g1}
V\, =\, \bigoplus_{i=1}^{2r} L_i
\end{equation}
such that
\begin{enumerate}
\item $\text{degree}(L_i)\, \geq\, \text{degree}(L_j)$ if $i\, \leq\, j$,

\item $L^*_i\,=\, L_{2r-i+1}$ for all $1\, \leq\, i\, \leq \, r$, and

\item for every $1\, \leq\, i\, <\, 2r$, the annihilator of the subbundle
$$
\bigoplus_{k=1}^i L_k \, \subset\, \bigoplus_{j=1}^{2r} L_j\,=\, V\, ,
$$
for the symplectic form $\varphi$, is the subbundle $\bigoplus_{k=1}^{2r-i} L_k$.
\end{enumerate}
The decomposition in \eqref{g1} is constructed by first splitting the Harder-Narasimhan filtration for $V$,
and then splitting each semistable bundle in the direct summand as a direct sum of holomorphic line bundles.

The criterion in Theorem \ref{thm1} says that there is a co-Higgs field
$$
\theta\, \in\, H^0(X,\, {\rm ad}(E_G)\otimes TX)
$$
such that $(E_G,\, \theta)$ is stable if and only if $\text{degree}(L_i)- \text{degree}(L_{i+1})\, \leq\, 2$
for all $1\, \leq\, i\, \leq\, r$.

Note that if $\text{degree}(L_i)- \text{degree}(L_{i+1})\, \leq\, 2$
for all $1\, \leq\, i\, \leq\, r$, then $\text{degree}(L_i)- \text{degree}(L_{i+1})\, \leq\, 2$
for all $1\, \leq\, i\, \leq\, 2r-1$, because $\text{degree}(L_{i+r})- \text{degree}(L_{i+r+1})\,
=\, \text{degree}(L_{r-i})- \text{degree}(L_{r-i+1})$ for all $1\, \leq\, i\, \leq\, r-1$.

\section{Strata for moduli}\label{SectStrata}

Apart from the case of $G=\mbox{SL}(2,\mathbb C)$ with odd degree, for which there exists a global description of the moduli space as a universal elliptic curve (cf. \cite[Section 7]{R0}), a neat global description of moduli for all $G$ remains to be found. The above arguments yield a description of strata in the moduli, the strata corresponding to the Harder-Narasimhan type or equivalently the splitting type of the bundle.

The Harder-Narasimhan type of the bundle corresponds to a dominant co-character $\mu: {\mathbb C}^*\longrightarrow T$, the maximal torus of the group. By taking a logarithm, this is in turn an element $M$ of the dual weight lattice in the Lie algebra $\mathfrak t$ of the torus, lying in the positive Weyl chamber: $\chi(M)\geq 0$ for all simple positive roots $\chi$ . The stability condition says that for all simple positive roots $\chi$, $\chi(M)\leq 2$. The set of cocharacters that this yields could be unbounded if the Lie algebra of the group has a non-trivial center, but we are considering bundles with a fixed topological degree; this cuts down the set of cocharacters to a finite set $\mathcal M$, and so we have a finite set of strata ${\mathcal S}_{ M},M\in \mathcal M$ in the co-Higgs moduli.

For a given cocharacter $M$, we have a finite dimensional vector space $V_M$ of possible co-Higgs fields in $H^0({\rm ad}(E_P)\otimes TX)= H^0({\rm ad}(E_P)\otimes \CO_X(2))$. Splitting the Lie algebra of $G$ into a torus $\mathfrak t$ and a sum of root spaces $R_\chi$, we obtain
$$ H^0({\rm ad}(E_P)\otimes TX)= {\mathbb C}^{3 rk(G)} + \sum_{\chi \in {\mathcal R}}H^0(\PP^1, \CO(\chi(M)+ 2).$$
Since $\sum_{\chi \in {\mathcal R}}\chi(M) = 0$, and $H^0(\PP^1, \CO(\chi(M)+ 2) = {\mathbb C}^{\chi(M)+3}$ if $ \chi(M)\geq -3$, $0$ otherwise, we find for the dimension of $V_M$:
$$ 3 \dim(G) + \sum_{\chi(M)>3}( \chi(M) +3)$$

The stability condition, as we have seen, is (Zariski) open, and the example in the previous section tells us that it is non-empty. We thus have a family $U_M\subset V_M$ of stable co-Higgs fields of the same dimension as $V_M$. Describing the set with any precision seems rather complicated.

To get the stratum ${\mathcal S}_{ M}$, one must quotient by the group $Aut_M$ of automorphisms of the bundle. This is modeled on the parabolic group $P$ corresponding to $M$; the Levi factor is the same as that of $P$, but for the unipotent piece, one replaces the complex line ${\mathbb{C}}_\chi$ corresponding to each root space by the space of sections $H^0(\PP^1, \CO(\chi(M))$. The Lie algebra $aut_M$ of $Aut_M$ is given by (as a vector space) by a complex vector space of dimension
$$rk(G) + \sum_{\chi(M)>-1} (\chi(M)+1) = \dim(G) + \sum_{\chi(M)> 1} (\chi(M)-1)$$
The group $Aut_M$ will act freely on a generic set of $V_M$: indeed, in $V_M$ one can choose nonzero entries in the root spaces, and one also has three components for each generator of the torus that one can choose freely. Choosing one of the torus components to be regular semisimple reduces one on the Lie algebra level to a torus, and then adding in some non-zero root space components gives a trivial stabilizer.

The quotient ${\mathcal S}_{ M} =U_M/Aut_M$ will then have dimension given by the difference of dimensions, 
$$ 2 \dim(G) -2\#\{\chi| \chi(M)>3\} - \sum_{1< \chi(M)\leq 3} (\chi(M)-1).$$
For the case of bundles of trivial degree, we have a generic stratum (corresponding to $M=0$) of dimension $2\dim(G)$.


\end{document}